\documentclass{article}

\usepackage[ngerman,english]{babel}

\usepackage[T1]{fontenc}
\usepackage{lmodern}
\usepackage{parskip}

\usepackage{amsfonts}
\usepackage{amssymb}
\usepackage{amsthm}
\usepackage{amsbsy}
\usepackage{latexsym}
\usepackage{bm}

\usepackage{subfigure}
\usepackage{enumerate}

\usepackage{graphicx}
\usepackage[dvipsnames]{xcolor}
\usepackage{epsfig}

\usepackage{textcomp}

\usepackage{float}

\usepackage[latin1]{inputenc}
\usepackage[d]{esvect} 
\DeclareMathAlphabet{\mathpzc}{OT1}{pzc}{m}{it}
\usepackage{upgreek}
\usepackage{mathtools}
\usepackage[mathscr]{eucal}
\usepackage{ushort}
\usepackage{array}
\usepackage{vmargin}

\newtheorem{theorem}{Theorem}
\newtheorem{lemma}{Lemma}
\newtheorem{corollary}{Corollary}
\newtheorem{remark}{Remark}

\newtheorem{definition}{Definition}

\newtheorem{algorithm}{Algorithm}

\title{Resolution Guarantees for the Reconstruction of Inclusions in Linear Elasticity Based on Monotonicity Methods}
\author{Sarah Eberle\thanks{eberle@math.uni-frankfurt.de, Institute of Mathematics,
Goethe-University Frankfurt, Frankfurt am Main, Germany (corresponding author)}
\and Bastian Harrach\thanks{harrach@math.uni-frankfurt.de, Institute of Mathematics,
Goethe-University Frankfurt, Frankfurt am Main, Germany}}

\date{}

\begin{document}

\maketitle

\begin{abstract}
\noindent
We deal with the reconstruction of inclusions in elastic bodies based on monotonicity methods and construct conditions
under which a resolution for a given partition can be achieved. These conditions take into account the background error 
as well as the measurement noise. As a main result, this shows us that the resolution guarantees depend heavily on the Lam\'e parameter $\mu$ and only marginally on $\lambda$.
\end{abstract}

{\bf Keywords}:
resolution guarantees, inverse problem, linear elasticity, detection and reconstruction of inclusions, monotonicity methods


\vspace{1cm}

\section{Introduction} 

In this paper we deal with the detection and reconstruction of inclusions in elastic bodies based on monotonicity methods, where the main focus lies on the so-called "resolution guarantees". Our results are of special importance, when considering simulations based on real data.
\\
\\
Before we introduce the definition of the resolution guarantees we state the setting.
Let $\Omega\subset \mathbb{R}^d$ ($d=2$ or $3$) be a bounded and connected open set 
with Lipschitz boundary $\partial\Omega=\Gamma=\overline{\Gamma_{\mathrm{D}}\cup \Gamma_{\mathrm{N}}}$, $\Gamma_{\mathrm{D}}\cap \Gamma_{\mathrm{N}}=\emptyset$,
where $\Gamma_{\textup D}$ and $\Gamma_{\textup N}$ are the corresponding Dirichlet and Neumann boundaries.
{\color{black} We assume that  $\Gamma_{\textup D}$ and $\Gamma_{\textup N}$ are relatively open and connected.} Further, $\Omega$ is divided into $N$ disjoint open subsets $(\omega_s)_{s=1,\ldots,N}$, i.e. a pixel partition, such that 
$\overline{\Omega}=\bigcup_{s=1,\ldots,N}\overline{\omega_s}$.
\\
\\
We base our considerations on the work \cite{Harrach_Resol}, where the resolution guarantees for the electrical impedance tomography (EIT) problem were considered.

\begin{definition}\label{resolution}
An inclusion detection method that yields a reconstruction $\mathcal{D}_R$ to the true inclusion $\mathcal{D}$ is
said to fulfill a resolution guarantee w.r.t. a partition $(\omega_s)_{s=1,\ldots,N}\subseteq \Omega$ if
\begin{itemize}
\item[(i)] $\omega_s\subseteq \mathcal{D}$ implies $\omega_s\subseteq\mathcal{D}_R$ for $s\in\lbrace 1,2,\ldots, N\rbrace$
\\
(i.e. every element that is covered by the inclusion will be marked in the reconstruction),
\item[(ii)] $\mathcal{D}=\emptyset$ implies $\mathcal{D}_R = \emptyset$
\\
(i.e. if there is no inclusion then no element will be marked in the reconstruction).
\end{itemize}
\end{definition}

\begin{remark}
Obviously, a reconstruction guarantee will not hold true for arbitrarily fine partitions. The achievable resolution will
depend on the number of applied boundary forces, the inclusion contrast, the background error and the measurement noise.
\end{remark}

\newpage

\noindent
We aim to construct conditions under which a resolution for a given partition can be achieved.  These conditions take into account the background error 
as well as the measurement noise. Thus, we show that it is generally possible to give rigorous guarantees in linear elasticity.
\\
\\
Before we go into more detail about the resolution guarantees, we would like to give an insight into the theory and methods of inclusion detection considered so far.
\\
\\
The theory of the inverse problem of linear elasticity, i.e.$\,$uniqueness results and Lipschitz stability studies, etc. were examined, e.g., in the works given below: 
\cite{imanuvilov2011reconstruction}, \cite{lin2017boundary}, \cite{nakamura1993identification} deal with the 2D case and \cite{ikehata1990inversion} with two and three dimensions. For uniqueness results in 3D we want to mention
\cite{eskin2002inverse} and \cite{nakamura1995inverse, nakamura2003global} as well as \cite{beretta2014lipschitz, beretta2014uniqueness} and \cite{lin2017boundary, nakamura1999layer,nakamura1995inverse}, where some boundary determination results were proved. In addition, results concerning the anisotropic case can be found, e.g., in \cite{Ikehata1998,
Ikehata2006} and \cite{Ikehata1999}.
Finally, \cite{Ikehata2002} discussed the reconstruction of inclusion from boundary measurements.
\\
\\
The following methods, among others, were used to solve the inverse problem of linear elasticity:
A Landweber iteration method was applied in \cite{Hubmer} and \cite{Marin_2005}. Further on, \cite{Jadamba} and \cite{Marin} considered regularization approaches. Beside the aforementioned methods, adjoint methods were used in \cite{Oberai_2004, Oberai_2003} and \cite{Seidl}. Further on, \cite{Andrieux}, \cite{Ferrier} and \cite{Steinhorst} took a look at reciprocity principles. Finally, we want to mention the monotonicity methods for linear elasticity developed by the authors of this paper in \cite{Eberle_mon_test}.
\\
\\
We focus on the monotonicity methods which are built on the examinations in \cite{Tamburrino06, tamburrino2002new}. These methods were first used for EIT (see, e.g., \cite{gebauer2008localized,harrach2009uniqueness,harrach2012simultaneous,harrach2018localizing, harrach2013monotonicity}) and then on other problems such as elasticity (see, e.g., \cite{Eberle_mon_test, Eberle_Mon_Bas_Reg, Eberle_Monotonicity}). In short, for this method the monotonicity properties of Neumann-to-Dirichlet operator plays an essential role. All in all, this builds the basis for our considerations.
\\
\\
This paper is organized as follows:
\\
First, we introduce the problem statement for linear elasticity and the setting, where we distinguish between the continuous and discrete case.
Next, we give a summary of the monotonicity methods, i.e., the standard monotonicity tests as well as the linearized monotonicity tests.
In Section 4, we present the background for the resolution guarantees and state the algorithms of the aforementioned monotonicity tests. Further on, we prove the required theorems which build the basis for the algorithms. Finally, we simulate the reconstruction for different settings. As a main result, we conclude that the resolution guarantees depend heavily on the Lam\'e parameter $\mu$ and only marginally on $\lambda$.


\section{Problem Statement and Setting}

We take a look at the continuous and discrete setting and introduce the corresponding problems as well the required assumptions concerning the inclusion, the background as well as the measurement error.

\subsection{Continuous case}

We start with the introduction of the problems of interest, e.g., the {\it direct} as well as the {\it inverse problem}
of stationary linear elasticity.
\\
\\
For the following, we define
\begin{align*} 
L_+^\infty(\Omega):=\lbrace w \in L^\infty(\Omega):\underset{x\in\Omega}{\text{ess\,inf}}\,w(x)>0\rbrace.
\end{align*}
\noindent
Let $u:\Omega\to\mathbb{R}^d$ be the displacement vector, $\mu,\lambda:\Omega\to L^{\infty}_+(\Omega)$ the Lam\'{e} parameters, 
$\hat{\nabla} u=\tfrac{1}{2}\left(\nabla u + (\nabla u)^T\right)$ the symmetric gradient, $n$ is \mbox{the normal}
vector pointing outside of $\Omega$ , $g\in L^{2}(\Gamma_{\textup N})^d$ the boundary force and $I$ the $d\times d$-identity matrix.
We define the divergence of a matrix $A\in \mathbb{R}^{d\times d}$ via 
$\nabla\cdot A=\sum\limits_{i,j=1}^d\dfrac{\partial A_{ij}}{\partial x_j}e_i$, where $e_i$ is a unit vector
and $x_j$ a component of a vector from $\mathbb{R}^d$.
\\
The boundary value problem of linear elasticity ({\it direct problem}) is 
{\color{black} that $u\in H^1(\Omega)^d$ solves}
\begin{align}\label{direct_1}
\nabla\cdot \left(\lambda (\nabla\cdot u)I + 2\mu \hat{\nabla} u \right)&=0 \quad \mathrm{in}\,\,\Omega,\\
\left(\lambda (\nabla\cdot u)I + 2\mu \hat{\nabla} u \right) n&=g\quad \mathrm{on}\,\, \Gamma_{\textup N},\label{direct_2}\\
u&=0 \quad \mathrm{on}\,\, \Gamma_{\textup D}.\label{direct_3}
\end{align}
\noindent
From a physical point of view, this means that we deal with an elastic test body $\Omega$ which is fixed (zero displacement)
at $\Gamma_{\mathrm{D}}$ (Dirichlet condition) and apply a force $g$ on $\Gamma_{\mathrm{N}}$ (Neumann condition).
This results in the displacement $u$, which is measured on the boundary $\Gamma_{\mathrm{N}}$.
\\
\\
\noindent
The equivalent weak formulation of the boundary value problem (\ref{direct_1})-(\ref{direct_3})  is
that $u\in\mathcal{V}$ fulfills
\begin{align}
\label{var-direct_1}
\int_{\Omega} 2 \mu\, \hat{\nabla}u : \hat{\nabla}v  + \lambda \nabla \cdot u \,\nabla\cdot  v\,dx=\int_{\Gamma_{\textup N}}g \cdot v \,ds \quad \text{ for all } v\in \mathcal{V},
\end{align}
where
$\mathcal{V}:=\left\{   v\in H^1(\Omega)^d:\,  v_{|_{\Gamma_{\textup D}}}=0\right\}$.
\\
\\
We want to remark that for $\lambda,\mu \in L^{\infty}_+(\Omega)$, the existence and uniqueness of a solution to the variational formulation (\ref{var-direct_1}) can be shown by
the Lax-Milgram theorem (see e.g., in \cite{Ciarlet}).

\subsubsection*{Neumann-to-Dirichlet operator and its monotonicity properties}

Measuring boundary displacements that result from applying forces to $\Gamma_{\textup{N}}$ can be modeled by the
Neumann-to-Dirichlet operator $\Lambda(\lambda,\mu)$ defined by
\begin{align*}
\Lambda(\lambda,\mu): L^2(\Gamma_{\textup N})^d\rightarrow L^2(\Gamma_{\textup N})^d: \quad  g\mapsto u_{|_{\Gamma_{\textup N}}},
\end{align*}
\noindent
where $u\in\mathcal{V}$ solves (\ref{direct_1})-(\ref{direct_3}).
\\
\\
This operator is self-adjoint, compact and linear
{\color{black} (see Corollary 1.1 from \cite{Eberle_mon_test}}).
Its associated bilinear form is given by
\begin{align}
\langle g,\Lambda(\lambda,\mu)h\rangle=\int_{\Omega} 2 \mu\, \hat{\nabla}u^g_{(\lambda,\mu)} : \hat{\nabla}u^h_{(\lambda,\mu)}  + 
\lambda \nabla \cdot u^g_{(\lambda,\mu)} \,\nabla\cdot  u^h_{(\lambda,\mu)}\,dx,\label{bilinear_Lambda}
\end{align}
\noindent
where $u_{(\lambda,\mu)}^g$ solves the problem (\ref{direct_1})-(\ref{direct_3}) and $u_{(\lambda,\mu)}^h$ 
the corresponding problem with boundary force $h\in L^2(\Gamma_{\mathrm{N}})^d$.
\noindent
\\
\\
Another important property of $\Lambda(\lambda,\mu)$ is its Fr\'echet differentiability (for the corresponding proof 
see {\color{black} Lemma 2.3 in} \cite{Eberle_mon_test}).
For directions $\hat{\lambda}, \hat{\mu}\in L^\infty(\Omega)$, the derivative
\begin{align*}
\Lambda'(\lambda,\mu)(\hat{\lambda},\hat{\mu}):  L^2(\Gamma_{\textup N})^d\rightarrow  L^2(\Gamma_{\textup N})^d
\end{align*}
is the self-adjoint compact linear operator associated to the bilinear form
\begin{align*}
\langle \Lambda'(\lambda,\mu)(\hat{\lambda},\hat{\mu})g,h\rangle
=&-\int_{\Omega} 2 \hat{\mu}\, \hat{\nabla}u^g_{(\lambda,\mu)} : \hat{\nabla}u^h_{(\lambda,\mu)}  + \hat{\lambda} \nabla \cdot u^g_{(\lambda,\mu)} \,\nabla\cdot  u^h_{(\lambda,\mu)}\,dx.
\end{align*}
Note that for  $\hat{\lambda}_0, \hat{\lambda}_1, \hat{\mu}_0, \hat{\mu}_1 \in L^\infty(\Omega)$ with
$\hat{\lambda}_0\leq \hat{\lambda}_1 \text{ and }  \hat{\mu}_0 \leq \hat{\mu}_1$
we obviously have
\begin{align}\label{mon_Frechet}
\Lambda'(\lambda,\mu)(\hat{\lambda}_0,\hat{\mu}_0)\geq \Lambda'(\lambda,\mu)(\hat{\lambda}_1,\hat{\mu}_1),
\end{align}
\noindent
in the sense of {\color{black} quadratic forms}.
\\
\\
The {\it inverse problem} we consider here is the following:
{\color{black}
\begin{align*}
\text{ \it  Find the support of } (\lambda-\lambda_0,\mu-\mu_0)^T \text{ \it knowing the Neumann-to-Dirichlet operator  } \Lambda(\lambda,\mu).  
 \end{align*}
}
\subsection{Discrete case}\label{assump}
Next, we go over to the discrete case and take a look at the bounded domain $\Omega\subset \mathbb{R}^d$ with piecewise smooth boundary representing the elastic object.
Further on, let $\lambda,\mu:\Omega\to \mathbb{R}^+$ be the real valued Lam\'e parameter distribution inside $\Omega$.
\\
\\
We apply the forces $g_l$ on the Neumann boundary of the object, where the location of their support is denoted by 
$\Gamma_{\mathrm{N}}^{(l)}\subseteq\Gamma_{\mathrm{N}}$, $l=1,\ldots,M$. We assume that the patches are disjoint. Thus, the discrete boundary value problem is given by
\begin{align}\label{bound_prob_1}
\nabla\cdot \left(\lambda(\nabla\cdot u)I+2\mu\hat{\nabla}u\right)&=0 \,\,\quad \text{in}\,\,\Omega,\\
\left(\lambda(\nabla\cdot u)I+2\mu\hat{\nabla} u\right) n&=g_l\quad \text{on}\,\,\Gamma_{\mathrm{N}}^{(l)}, \label{bound_prob_2}\\
\left(\lambda(\nabla\cdot u)I+2\mu\hat{\nabla} u\right) n&=0  \quad \text{on}\,\,\Gamma_{\mathrm{N}}^{(i)}, \, i\neq l, \label{bound_prob_3}\\
u&=0  \,\,\quad \text{in}\,\,\Gamma_{\mathrm{D}}.\label{bound_prob_4}
\end{align}
\noindent
\\
The resulting displacement measurements are represented by the discrete version of $\Lambda(\lambda,\mu)$:
\begin{align}\label{NtD}
{\boldsymbol\Lambda}(\lambda,\mu)=\left(\Lambda_{l}^{(k)}(\lambda,\mu)\right)_{k,l=1,\ldots,M}
\end{align}
\noindent
with
\begin{align*}
\Lambda_l^{(k)}(\lambda,\mu):=\int_{\Gamma_{\mathrm{N}}^{(l)}}g_l\cdot u^{(k)}\,ds
\end{align*}
\noindent
and $u^{(k)}$ solves the boundary value problem (\ref{bound_prob_1})-(\ref{bound_prob_4}) for a boundary load $g_k$.

\subsection*{Assumptions regarding the inclusion, the background as well as the measurement error}

In the following, we introduce our assumptions and definitions concerning the Lam\'e parameters for the inclusion and background
including their error considerations.

\begin{itemize}
\item[(a)] Distribution of Lam\'e parameter $(\lambda(x),\mu(x))$:
\begin{align*}
(\lambda(x),\mu(x))=
\begin{cases}
(\lambda_{\mathcal{D}}(x),\mu_{\mathcal{D}}(x)),&\quad x\in\mathcal{D},\\
(\lambda_{\mathcal{B}}(x),\mu_{\mathcal{B}}(x)),&\quad x\in\Omega\setminus\mathcal{D},
\end{cases}
\end{align*}
\noindent
where $\mathcal{D}$ denotes the unknown inclusion and $\mathcal{B}$ the background.

\item[(b)] Background error $\epsilon^\lambda,\epsilon^\mu\geq 0$:
\begin{align*}
|(\lambda_{\mathcal{B}}(x),\mu_{\mathcal{B}}(x))-(\lambda_0,\mu_{0})|
\leq(\lambda_0\epsilon^\lambda,\mu_0\epsilon^\mu)\quad \text{for all}\,x\in\Omega\setminus\mathcal{D},
\end{align*}
\noindent
where the background Lam\'e parameters $\lambda_{\mathcal{B}}(x)$ and $\mu_{\mathcal{B}}(x)$ approximately agree with known positive constants $\lambda_0$ and $\mu_0$.

\item[(c)] Inclusion contrast $c^{\lambda}, c^{\mu}>0$: 
\\
We distinguish between the following two cases

 either $(\lambda_{\mathcal{D}}(x),\mu_{\mathcal{D}}(x))-(\lambda_0,\mu_0)\geq (c^{\lambda},c^{\mu})\,\,
\text{for all}\, x\in\mathcal{D}$
\\\
or $\quad\,\,\,\,(\lambda_0,\mu_0)-(\lambda_{\mathcal{D}}(x),\mu_{\mathcal{D}}(x))\geq (c^{\lambda},c^{\mu})\,\,
\text{for all}\, x\in\mathcal{D}$,
\\
\\
where the lower bounds $c^\lambda$ and $c^\mu$ are known.
\item[(d)] Measurement noise $\delta\geq 0$:
\begin{align*}
|| {\boldsymbol\Lambda}(\lambda,\mu)-{\boldsymbol\Lambda}^\delta(\lambda,\mu)||\leq \delta,
\end{align*}
\noindent
i.e., we assume that ${\boldsymbol\Lambda(\lambda,\mu)}$ is determined up to noise level $\delta>0$.
\end{itemize}


\section{Summary of the Monotonicity Methods}

First, we state the monotonicity estimates for the Neumann-to-Dirichlet operator $\Lambda(\lambda,\mu)$
{\color{black}
and denote by $u^g_{(\lambda,\mu)}$ the solution of problem (\ref{direct_1})-(\ref{direct_3}) for the boundary load
$g$ and the Lam\'e parameters $\lambda$ and $\mu$.
}
\begin{lemma}[Lemma 3.1 from \cite{Eberle_Monotonicity}]
\label{mono}
Let {\color{black}$(\lambda_1,\mu_1),(\lambda_2,\mu_2)  \in  L_+^\infty(\Omega)\times L_+^\infty(\Omega) $},  $g\in L^2(\Gamma_{\textup N})^d$ be an applied boundary force, and let 
$u_1:=u^{g}_{(\lambda_1,\mu_1)}\in \mathcal{V}$, $u_2:=u^{g}_{(\lambda_2,\mu_2)}\in \mathcal{V}$. Then
\begin{align}
\label{eqmono}
\int_\Omega &2(\mu_1-\mu_2)\hat{\nabla}u_2:\hat{\nabla}u_2+(\lambda_1-\lambda_2)\nabla\cdot u_2\nabla\cdot u_2\,dx\\ \nonumber
&\geq \langle g,\Lambda(\lambda_2,\mu_2)g\rangle-\langle g,\Lambda(\lambda_1,\mu_1)g\rangle\\ 
&\geq \int_\Omega 2(\mu_1-\mu_2)\hat{\nabla}u_1 :\hat{\nabla}u_1+(\lambda_1-\lambda_2)\nabla\cdot u_1 \nabla\cdot u_1\,dx.
\end{align}
\end{lemma}

\noindent

\begin{lemma}[Lemma 2 from \cite{Eberle_mon_test}]\label{mono_2}
Let {\color{black}$(\lambda_1,\mu_1),(\lambda_2,\mu_2)  \in  L_+^\infty(\Omega)\times L_+^\infty(\Omega) $},   $g\in L^2(\Gamma_{\textup N})^d$ be an applied boundary force, and let 
\mbox{$u_1:=u^{g}_{(\lambda_1,\mu_1)}\in \mathcal{V}$}, $u_2:=u^{g}_{(\lambda_2,\mu_2)}\in \mathcal{V}$. Then
\begin{align}\label{mono_lame}
\langle g&,\Lambda(\lambda_2,\mu_2)g\rangle-\langle g,\Lambda(\lambda_1,\mu_1)g\rangle\\  \nonumber
&\geq 
\int_{\Omega} 2\left(\mu_2-\frac{\mu_2^2}{\mu_1}\right)\hat{\nabla}u_2:\hat{\nabla}u_2\,dx
+\int_{\Omega}\left(\lambda_2-\frac{\lambda_2^2}{\lambda_1}\right)\nabla\cdot u_2 \nabla \cdot u_2\,dx\\
&=
\int_{\Omega} 2\frac{\mu_2}{\mu_1}\left(\mu_1-\mu_2\right)\hat{\nabla}u_2:\hat{\nabla}u_2\,dx
+\int_{\Omega}\frac{\lambda_2}{\lambda_1}\left(\lambda_1-\lambda_2\right) \nabla\cdot u_2 \nabla \cdot u_2\,dx.
\end{align}
\end{lemma}

\noindent

\begin{corollary}[Corollary 3.2 from \cite{Eberle_Monotonicity}]\label{monotonicity}
For $(\lambda_0,\mu_0),(\lambda_1,\mu_1) \in L_+^\infty(\Omega)\times L_+^\infty(\Omega) $
\begin{align}
\lambda_0\leq \lambda_1 \text{ and } \mu_0\leq \mu_1 \quad \text{  implies } \quad \Lambda(\lambda_0,\mu_0)\geq \Lambda(\lambda_1,\mu_1).
\end{align}
\end{corollary}
\noindent
We give a short overview concerning the monotonicity methods, where  we restrict ourselves to the case $\lambda_1\geq \lambda_0$, $\mu_1\geq \mu_0$.
In the following, let $\mathcal{D}$ be the unknown inclusion and $\chi_\mathcal{D}$ the characteristic function w.r.t. $\mathcal{D}$. In addition, we deal with "noisy difference measurements",
i.e. distance measurements between $u^g_{(\lambda,\mu)}$ and $u^g_{(\lambda_0,\mu_0)}$ affected by noise,
which stem from  the corresponding system (\ref{direct_1})-(\ref{direct_3}).
\\
\\
{\color{black}
We define the outer support in correspondence to \cite{Eberle_mon_test} as follows: 
let $\phi=(\phi_1,\phi_2):\Omega\to \mathbb{R}^2$ be a measurable
function,
the outer support $\underset{\partial\Omega}{\mathrm{out}}\,\mathrm{supp}(\phi)$ is the 
 complement (in $\overline{\Omega}$) of the union of those relatively open $U\subseteq\overline{\Omega}$
 that are connected to $\partial\Omega$ and for which $\phi\vert_{U}=0$,
\noindent
respectively.

\subsection{Standard Monotonicity Tests}

We start our consideration with the standard monotonicity tests and take a look at the case for exact as well as noisy data. Here, we denote the material without inclusion by $(\lambda_0,\mu_0)$ and the Lam\'e parameters of the inclusion by $(\lambda_1,\mu_1)$.

\subsubsection*{Tests for exact and noisy data }

\begin{corollary}{Standard monotonicity test }(Corollary 2.4 from \cite{Eberle_mon_test})
\\
Let $\lambda_0,\lambda_1,\mu_0,\mu_1\in\mathbb{R}^+$, $(\lambda,\mu)=(\lambda_0+(\lambda_1-\lambda_0)\chi_\mathcal{D},\mu_0+(\mu_1-\mu_0)\chi_{\mathcal{D}})$
with $\lambda_1>\lambda_0$ and $\mu_1>\mu_0$ and the inclusion $\mathcal{D}=\mathrm{out}_{\partial\Omega}\,\mathrm{supp}((\lambda-\lambda_0,\mu-\mu_0)^T)$.
Further on, let $\alpha^{\lambda},\alpha^{\mu}\geq 0$, $\alpha^{\lambda}+\alpha^{\mu}>0$ with $\alpha^{\lambda} \leq \lambda_1 -\lambda_0$, $\alpha^{\mu}\leq \mu_1 -\mu_0$.
Then for every open set $\omega\subseteq\Omega$
\begin{align*}
\omega\subseteq\mathcal{D}\quad\text{if\,and\,only\,if}\quad\Lambda(\lambda_0+\alpha^{\lambda}\chi_\omega,\mu_0+\alpha^{\mu}\chi_\omega)\geq\Lambda(\lambda,\mu).
\end{align*}
\end{corollary}

\begin{corollary}{Standard monotonicity test for noisy data }(Corollary 2.6 from \cite{Eberle_mon_test})
\\
Let $\lambda_0,\lambda_1,\mu_0,\mu_1\in\mathbb{R}^+$, $(\lambda,\mu)=(\lambda_0+(\lambda_1-\lambda_0)\chi_\mathcal{D},\mu_0+(\mu_1-\mu_0)\chi_{\mathcal{D}})$
with $\lambda_1>\lambda_0$ and $\mu_1>\mu_0$ and the inclusion $\mathcal{D}=\mathrm{out}_{\partial\Omega}\,\mathrm{supp}((\lambda-\lambda_0,\mu-\mu_0)^T)$.
Further on, let $\alpha^{\lambda}\,\alpha^{\mu}\geq 0$, $\alpha+\beta>0$ with $\alpha^{\lambda} \leq \lambda_1 -\lambda_0$, $\alpha^{\mu}\leq \mu_1 -\mu_0$ and let each noise level $\delta>0$ fulfill
\begin{align*}
 \Vert{\Lambda^\delta(\lambda,\mu) - \Lambda(\lambda,\mu)}\Vert< \delta.
\end{align*}
\noindent
Then for every open set $\omega\subseteq\Omega$ there exists a noise level $\delta_0>0$,
such that $\omega$ is correctly detected as inside the inclusion $\mathcal{D}$
by the condition
\begin{align*}
\omega\subseteq\mathcal{D}\quad \textnormal{ if and only if} \quad
\Lambda(\lambda_0+\alpha^{\lambda}\chi_{\omega},\mu_0+\alpha^{\mu}\chi_{\omega})
-\Lambda^\delta(\lambda,\mu)
+\delta I
\geq 0
\end{align*}
\noindent
for all $\quad 0<\delta< \delta_0$.
\end{corollary}

\subsection{Linearized Monotonicity Tests}

We also introduce the linearized monotonicity tests as a modification of the standard methods. Similar as before, we deal with the exact as well as perturbed problem.

\subsubsection*{Tests for exact and noisy data }

\begin{corollary}{Linearized monotonicity test }(Corollary 2.7 from \cite{Eberle_mon_test}\label{lin_mon_1})
\\
Let $\lambda_0$, $\lambda_1$, $\mu_0$, $\mu_1\in\mathbb{R}^+$ with $\lambda_1>\lambda_0$, $\mu_1>\mu_0$  
and assume that 
$(\lambda,\mu)=(\lambda_0+(\lambda_1-\lambda_0)\chi_\mathcal{D},\mu_0+(\mu_1-\mu_0)\chi_{\mathcal{D}})$
with $\mathcal{D}=\mathrm{out}_{\partial\Omega}\,\mathrm{supp}((\lambda-\lambda_0,\mu-\mu_0)^T)$.
Further on let $\alpha^\lambda,\alpha^\mu\geq 0$, $\alpha^\lambda+\alpha^\mu>0$ and $\alpha^\lambda \leq \tfrac{\lambda_0}{\lambda_1}(\lambda_1 -\lambda_0)$,  $\alpha^\mu \leq\tfrac{\mu_0}{\mu_1} (\mu_1 -\mu_ 0)$.
Then for every open set $\omega$
\begin{align*}
\omega\subseteq\mathcal{D}\quad\text{if and only if}\quad \Lambda(\lambda_0,\mu_0)+\Lambda^\prime(\lambda_0,\mu_0)(\alpha^\lambda \chi_\omega, \alpha^\mu \chi_\omega)\geq \Lambda(\lambda,\mu).
\end{align*}
\end{corollary}
\noindent

\begin{corollary}{Linearized monotonicity test for noisy data }(Corollary 2.9 from \cite{Eberle_mon_test}\label{lin_mono_test_noise})
\\
Let $\lambda_0$, $\lambda_1$, $\mu_0$, $\mu_1\in\mathbb{R}^+$ with $\lambda_1>\lambda_0$, $\mu_1>\mu_0$  
and assume that 
$(\lambda,\mu)=(\lambda_0+(\lambda_1-\lambda_0)\chi_\mathcal{D},\mu_0+(\mu_1-\mu_0)\chi_{\mathcal{D}})$
with $\mathcal{D}=\mathrm{out}_{\partial\Omega}\,\mathrm{supp}((\lambda-\lambda_0,\mu-\mu_0)^T)$.
Further on, let $\alpha^\lambda,\alpha^\mu\geq 0$, $\alpha^\lambda+\alpha^\mu>0$ with $\alpha^\lambda \leq \frac{\lambda_0}{\lambda_1} (\lambda_1 -\lambda_0) $,
$\alpha^\mu\leq \frac{\mu_0}{\mu_1} (\mu_1 -\mu_0) $.
Let $\Lambda^\delta$ be the Neumann-to-Dirichlet operator for noisy difference measurements with noise level $\delta>0$.
Then for every open set $\omega\subseteq\Omega$ there exists a noise level $\delta_0>0$, such that for all
$0<\delta<\delta_0$, $\omega$ is correctly detected as inside or not inside the inclusion $\mathcal{D}$ 
by the following monotonicity test
\begin{eqnarray*}
\omega\subseteq\mathcal{D}\quad \textnormal{\it if and only if} \quad
\Lambda(\lambda_0,\mu_0) +\Lambda^\prime(\lambda_0,\mu_0)(\alpha^\lambda\chi_{\omega},\alpha^\mu\chi_{\omega})-\Lambda^\delta(\lambda,\mu)+\delta I \geq 0.
\end{eqnarray*}
\end{corollary}


\section{Resolution Guarantees}

In this section we formulate the algorithms for the monotonicity tests, i.e., the standard monotonicity tests
as well as the linearized tests and follow the considerations in \cite{Harrach_Resol}, where resolution guarantees for EIT were analysed.

\subsection{Algorithms}

Before we take a look at the algorithms for the reconstruction, we define the corresponding notations which we will use in the following. We set
\begin{align*}
(\lambda_{\mathcal{B}_{\min}},\mu_{\mathcal{B}_{\min}})&:=(\lambda_0(1-\epsilon^\lambda),\mu_0(1-\epsilon^\mu)),\\
(\lambda_{\mathcal{B}_{\max}},\mu_{\mathcal{B}_{\max}})&:=(\lambda_0(1+\epsilon^\lambda),\mu_0(1+\epsilon^\mu)),\\
(\lambda_{\mathcal{D}_{\min}},\mu_{\mathcal{D}_{\min}})&:=(\lambda_0+c^{\lambda},\mu_0+c^{\mu}),\\
(\lambda_{\mathcal{D}_{\max}},\mu_{\mathcal{D}_{\max}})&:=(\lambda_0-c^{\lambda},\mu_0-c^{\mu}),
\end{align*}
\noindent
where the quantities are given in Subsection \ref{assump} assumptions (a)-(d).

\subsubsection{Algorithms for standard monotonicity tests}

We now formulate the algorithms for the standard monotonicity tests. We start with the case that 
\begin{align*}
\left(\lambda_{\mathcal{D}_{\min}},\mu_{\mathcal{D}_{\min}}\right)> (\lambda_{\mathcal{B}_{\max}},\mu_{\mathcal{B}_{\max}})
\end{align*}
\noindent
and it holds that
\begin{align*}
\left( \lambda_{\mathcal{D}}-\lambda_0,\mu_{\mathcal{D}}-\mu_0\right)\geq \left(c^{\lambda},c^{\mu}\right).
\end{align*}

\begin{algorithm}\label{alg_stand_test}
Mark each resolution element $\omega_s$ for which
\begin{align*}
{\boldsymbol\Lambda}\left( \tau_s^{\lambda},\tau_s^{\mu}\right)+\delta {\bf I}\geq {\boldsymbol\Lambda}^\delta(\lambda,\mu),\quad s\in\lbrace 1,2,\ldots,N\rbrace,
\end{align*}
\noindent
where
\begin{align*}
\tau_s^{\lambda}&:=\lambda_{\mathcal{B}_{\min}}\chi_{\Omega\setminus \omega_s}+\lambda_{\mathcal{D}_{\min}}\chi_{\omega_s},\\
\tau_s^{\mu}&:=\mu_{\mathcal{B}_{\min}}\chi_{\Omega\setminus \omega_s} +\mu_{\mathcal{D}_{\min}}\chi_{\omega_s}.
\end{align*}
\noindent
Then the reconstruction $\mathcal{D}_R$ is given by the union of the marked resolution elements.
\end{algorithm}
\noindent
\\
Further on, we take a look at the special case for "weaker" Lam\'e parameter inclusions and consider the case
\begin{align}\label{weak_param}
(\lambda_{\mathcal{D}_{\max}},\mu_{\mathcal{D}_{\max}})<(\lambda_{\mathcal{B}_{\min}},\mu_{\mathcal{B}_{\min}}).
\end{align}

\begin{algorithm}\label{alg_stand_test_weak}
Mark each resolution element $\omega_s$ for which
\begin{align*}
{\boldsymbol\Lambda}\left(\tau_s^{\lambda},\tau_s^{\mu}\right)-\delta {\bf I}\leq {\boldsymbol\Lambda}^\delta(\lambda,\mu),\quad s\in\lbrace 1,2,\ldots,N\rbrace,
\end{align*}
\noindent
where
\begin{align*}
\tau_s^{\lambda}&:= \lambda_{\mathcal{B}_{\max}}\chi_{\Omega\setminus \omega_s}+\lambda_{\mathcal{D}_{\max}}\chi_{\omega_s},\\
\tau_s^{\mu}&:=\mu_{\mathcal{B}_{\max}}\chi_{\Omega\setminus \omega_s} +\mu_{\mathcal{D}_{\max}}\chi_{\omega_s}.
\end{align*}
\end{algorithm}

\subsubsection{Algorithms for linearized monotonicity tests}

Replacing the monotonicity test for the case
$\left(\lambda_{\mathcal{D}_{\min}},\mu_{\mathcal{D}_{\min}}\right)> (\lambda_{\mathcal{B}_{\max}},\mu_{\mathcal{B}_{\max}})$, i.e.
\begin{align*}
{\boldsymbol\Lambda}\left( \tau_s^{\lambda},\tau_s^{\mu}\right)+\delta {\bf I}\geq {\boldsymbol\Lambda}^\delta(\lambda,\mu)
\end{align*}
\noindent
with their linearized approximations yields the linearized monotonicity test
\begin{align*}
{\boldsymbol\Lambda}\left( \lambda_{\mathcal{B}_{\min}},\mu_{\mathcal{B}_{\min}}\right)+{\boldsymbol\Lambda}^\prime\left( \lambda_{\mathcal{B}_{\min}},\mu_{\mathcal{B}_{\min}}\right)(\kappa^{\lambda} \chi_{\omega_s},\kappa^{\mu}\chi_{\omega_s})
+\delta {\bf I}\geq {\boldsymbol\Lambda}^\delta(\lambda,\mu),
\end{align*}
\noindent
where $\kappa^{\lambda},\kappa^{\mu}\in\mathbb{R}$ is a suitable contrast level defined in the following algorithm. Further, we assume the $\lambda_{\max}$ and $\mu_{\max}$ are global bounds with
\begin{align*}
\lambda(x)&\leq \lambda_{\max},\\
\mu(x)&\leq \mu_{\max}
\end{align*}
\noindent
for all $x\in\Omega$.

\begin{algorithm}\label{alg_lin_test}
Mark each resolution element $\omega_s$ for which
\begin{align*}
{\bf T}_s+\delta {\bf I}\geq {\boldsymbol\Lambda}^\delta(\lambda,\mu),
\end{align*}
\noindent
where
\begin{align*}
{\bf T}_s:={\boldsymbol\Lambda}\left( \lambda_{\mathcal{B}_{\min}},\mu_{\mathcal{B}_{\min}}\right)+{\boldsymbol\Lambda}^\prime\left( \lambda_{\mathcal{B}_{\min}},\mu_{\mathcal{B}_{\min}}\right)(\kappa^{\lambda}\chi_{\omega_s},\kappa^{\mu}\chi_{\omega_s}),
\end{align*}
\noindent
with
\begin{align}\label{kappa_l}
\kappa^{\lambda}&:=(c^{\lambda}+\lambda_0\epsilon^\lambda)\dfrac{\lambda_{\mathcal{B}_{\min}}}{\lambda_{\max}},\\
\kappa^{\mu}&:=(c^{\mu}+\mu_0\epsilon^\mu)\dfrac{\mu_{\mathcal{B}_{\min}}}{\mu_{\max}}.
\end{align}
\noindent
Then the reconstruction $\mathcal{D}_R$ is given by the union of the marked resolution elements.
\end{algorithm}
\noindent
\\
As for the standard monotonicity test, we formulate the linearized test for inclusions with weaker Lam\'e parameter which fulfill $(\lambda_{\mathcal{D}_{\max}},\mu_{\mathcal{D}_{\max}})<(\lambda_{\mathcal{B}_{\min}},\mu_{\mathcal{B}_{\min}})$.

\begin{algorithm}\label{alg_lin_test_weak}
Mark each resolution element $\omega_s$ for which
\begin{align*}
{\bf T}_s-\delta {\bf I}\leq {\boldsymbol\Lambda}^\delta(\lambda,\mu),
\end{align*}
\noindent
where
\begin{align*}
{\bf T}_s:={\boldsymbol\Lambda}\left( \lambda_{\mathcal{B}_{\max}},\mu_{\mathcal{B}_{\max}}\right)+{\boldsymbol\Lambda}^\prime\left( \lambda_{\mathcal{B}_{\max}},\mu_{\mathcal{B}_{\max}}\right)(\kappa^{\lambda}\chi_{\omega_s},\kappa^{\mu}\chi_{\omega_s})
\end{align*}
\noindent
with
\begin{align}\label{kappa_l_weak}
\kappa^{\lambda}&:=-(c^{\lambda}+\lambda_0\epsilon^\lambda),\\
\kappa^{\mu}&:=-(c^{\mu}+\mu_0\epsilon^\mu)
\end{align}
\noindent
for $s\in\lbrace 1,2,\ldots,N\rbrace$.
\noindent
\\
\\
Then the reconstruction $\mathcal{D}_R$ is given by the union of the marked resolution elements.
\end{algorithm}

\subsection{Formulation of theorems}

We will analyse the algorithms in more detail and take a look at the required theorems.

\subsubsection{Theorems for standard monotonicity tests}

\begin{theorem}\label{resol_stand_test_strong}
The reconstruction of Algorithm \ref{alg_stand_test} fulfills the resolution guarantees if 
\begin{align*}
\nu<-2\delta\leq 0
\end{align*} 
\noindent
with
\begin{align*}
\nu:=\max_{s=1,\ldots,N}\left( \min\left(\mathrm{eig}\left[ {\boldsymbol\Lambda}\left( \tau_s^{\lambda},\tau_s^{\mu}\right) -{\boldsymbol\Lambda}(\lambda_{\mathcal{B}_{\max}},\mu_{\mathcal{B}_{\max}})\right] \right)\right).
\end{align*}
\end{theorem}

\begin{proof}
We start with the consideration of part (i) from Definition \ref{resolution} and let $\omega_s\subseteq\mathcal{D}$. Then
\begin{align*}
(\tau_s^\lambda,\tau_s^\mu)=\left(\lambda_{\mathcal{B}_{\min}}\chi_{\Omega\setminus \omega_s}+\lambda_{\mathcal{D}_{\min}}\chi_{\omega_s},\mu_{\mathcal{B}_{\min}}\chi_{\Omega\setminus \omega_s} +\mu_{\mathcal{D}_{\min}}\chi_{\omega_s}\right)\leq (\lambda,\mu).
\end{align*}
\noindent 
The knowledge, that from $(\lambda_1,\mu_1)\leq (\lambda_2,\mu_2)$ it follows that
\begin{align}\label{relation_l_m}
{\boldsymbol\Lambda}(\lambda_1,\mu_1)\geq {\boldsymbol\Lambda}(\lambda_2,\mu_2),
\end{align}
\noindent
implies that
\begin{align*}
{\boldsymbol\Lambda}\left(\tau_s^\lambda,\tau_s^\mu\right)\geq {\boldsymbol\Lambda}(\lambda,\mu).
\end{align*}
\noindent
Hence,
\begin{align*}
{\boldsymbol\Lambda}\left( \tau_s^\lambda,\tau_s^\mu\right)+\delta{\bf I}\geq {\boldsymbol\Lambda}^\delta(\lambda,\mu),
\end{align*}
\noindent
so that $\omega_s$ will be marked by the algorithm.
\\
This shows that part (i) of the resolution guarantee is satisfied.
\\
\\
To prove part (ii) of the resolution guarantee, assume that $\mathcal{D}=\emptyset$ and $\mathcal{D}_R\neq \emptyset$.
Then there must be an index $s\in\lbrace 1,2,\ldots,N\rbrace$ with
\begin{align*}
{\boldsymbol\Lambda}\left( \tau_s^\lambda,\tau_s^\mu\right)+\delta {\bf I}\geq {\boldsymbol\Lambda}^\delta(\lambda,\mu).
\end{align*}
\noindent
Again, with the monotonicity relation (\ref{relation_l_m}), we obtain
\begin{align*}
-2\delta {\bf I}
&\leq  {\boldsymbol\Lambda}\left(\tau_s^\lambda, \tau_s^\mu \right)
-\left(\delta {\bf I}+{\boldsymbol\Lambda}^\delta(\lambda,\mu)\right)\\
&\leq {\boldsymbol\Lambda}\left( \tau_s^\lambda, \tau_s^\mu \right)-{\boldsymbol\Lambda}(\lambda,\mu)\\
&\leq {\boldsymbol\Lambda}\left( \tau_s^\lambda, \tau_s^\mu \right)-
{\boldsymbol\Lambda}(\lambda_{\mathcal{B}_{\max}},\mu_{\mathcal{B}_{\max}})
\end{align*}
and thus $\nu\geq -2\delta$, which is a contradiction.
\end{proof}
\noindent
\\
All in all, this theorem gives a rigorous yet conceptually simple criterion to check whether a given resolution guarantee is valid or not.

\begin{remark}
Given a partition $\left(\omega_s\right)_{s=1,\ldots,N}$ and bounds on the background, we obtain $\nu$ from calculating
\begin{align*}
{\boldsymbol\Lambda}\left( \tau_s^\lambda, \tau_s^\mu \right)
\quad\text{and}\quad
{\boldsymbol\Lambda}(\lambda_{\mathcal{B}_{\max}},\mu_{\mathcal{B}_{\max}})
\end{align*}
\noindent
by solving the boundary value problem (\ref{bound_prob_1})-(\ref{bound_prob_4}). If this yields a negative value for $\nu$,
then the resolution guarantee holds true up to a measurement error of $\delta<-\dfrac{\nu}{2}$.
\end{remark}
\noindent
\\
Next, we formulate the corresponding theorem for case (\ref{weak_param}).

\begin{theorem}
The reconstruction of Algorithm \ref{alg_stand_test_weak} fulfills the resolution guarantee if
\begin{align*}
\nu > 2\delta \geq 0
\end{align*}
\noindent
with
\begin{align*}
\nu:=\min_{s=1,\ldots,N}\left( \max\left( \mathrm{eig}\left[ {\boldsymbol\Lambda}\left(\tau_s^\lambda,\tau_s^\mu\right) -{\boldsymbol\Lambda}(\lambda_{\mathcal{B}_{\min}},\mu_{\mathcal{B}_{\min}})\right] \right)\right).
\end{align*}
\end{theorem}

\begin{proof}
The proof of part (i) of the resolution guarantee is analogous to the proof of part (i) in the theorem before. 
\\
To show part (ii) of the resolution guarantee, assume that $\mathcal{D}=\emptyset$ and $\mathcal{D}_R\neq\emptyset$. Then there must be an index $s \in \lbrace 1,2\ldots,N\rbrace$ with
\begin{align*}
&{\boldsymbol\Lambda}\left( \tau_s^\lambda, \tau_s^\mu \right)-\delta {\bf I}\\
&\leq {\boldsymbol\Lambda}^\delta(\lambda,\mu)\\
&\leq {\boldsymbol\Lambda}(\lambda,\mu)+\delta {\bf I}.
\end{align*}
\noindent
Using the results from before, we obtain
\begin{align*}
&{\boldsymbol\Lambda}\left(\tau_s^\lambda, \tau_s^\mu \right)-2\delta {\bf I}\\
&\leq {\boldsymbol\Lambda}(\lambda_{\mathcal{B}_{\min}},\mu_{\mathcal{B}_{\min}})
\end{align*}
\noindent
and thus $\nu\geq 2\delta$, which is a contradiction.
\end{proof}

\subsubsection{Theorems for linearized monotonicity tests}

\begin{theorem}\label{cal_max}
The reconstruction of Algorithm \ref{alg_lin_test} fulfills the resolution guarantee if 
\begin{align*}
\nu<-2\delta \leq 0
\end{align*}
\noindent
with
\begin{align*}
\nu:=\max_{s=1,\ldots,N}\left(\min\left(\mathrm{eig}\left[{\bf T}_s-{\boldsymbol\Lambda}(\lambda_{\mathcal{B}_{\max}},\mu_{\mathcal{B}_{\max}})\right]\right)\right).
\end{align*}
\end{theorem}

\begin{proof}
First, let $\omega_s\subseteq\mathcal{D}$ and let $g_j\in\mathbb{R}^M$, $j\in \lbrace 1, 2,\ldots ,N\rbrace$. In a body with interior Lam\'e parameters $(\lambda_{\mathcal{B}_{\min}},\mu_{\mathcal{B}_{\min}})$, let $u_{g_j}$ be the corresponding  displacements resulting from applying the boundary load $g_j$ to the $j$-th patch.
\begin{align*}
&g_j^T\left( {\boldsymbol\Lambda}(\lambda_{\mathcal{B}_{\min}},\mu_{\mathcal{B}_{\min}})-{\boldsymbol\Lambda}(\lambda,\mu)\right)g_j\\
&\geq \int_{\Omega}\dfrac{\lambda_{\mathcal{B}_{\min}}}{\lambda} (\lambda-\lambda_{\mathcal{B}_{\min}})\nabla \cdot u_{g_j} \nabla \cdot u_{g_j}
+ 2\dfrac{\mu_{\mathcal{B}_{\min}}}{\mu}(\mu-\mu_{\mathcal{B}_{\min}})\hat{\nabla}u_{g_j} : \hat{\nabla}u_{g_j}\,dx
\end{align*}
\noindent
and since $\omega_s\subseteq \mathcal{D}$ implies $\lambda-\lambda_{\mathcal{B}_{\min}}\geq (c^
{\lambda}+\lambda_0\epsilon^\lambda)\chi_{\omega_s}$ and $\mu-\mu_{\mathcal{B}_{\min}}\geq (c^{\mu}+\mu_0\epsilon^\mu)\chi_{\omega_s}$, it follows that
\begin{align*}
{\boldsymbol\Lambda}(\lambda_{\mathcal{B}_{\min}},\mu_{\mathcal{B}_{\min}})-{\boldsymbol\Lambda}(\lambda,\mu) \geq -{\boldsymbol\Lambda}^\prime(\lambda_{\mathcal{B}_{\min}},\mu_{\mathcal{B}_{\min}})(\kappa^{\lambda}\chi_{\omega_s},\kappa^{\mu}\chi_{\omega_s}).
\end{align*}
\noindent
Hence, we obtain that
\begin{align*}
&{\bf T}_s+\delta {\bf I}\\
&={\boldsymbol\Lambda}(\lambda_{\mathcal{B}_{\min}},\mu_{\mathcal{B}_{\min}})+{\boldsymbol\Lambda}^\prime(\lambda_{\mathcal{B}_{\min}},\mu_{\mathcal{B}_{\min}})(\kappa^{\lambda}\chi_{\omega_s},\kappa^{\mu}\chi_{\omega_s})+\delta {\bf I}\\
&\geq {\boldsymbol\Lambda}(\lambda,\mu) +\delta {\bf I}\\
&\geq {\boldsymbol\Lambda}^\delta(\lambda,\mu).
\end{align*}
\noindent
For the proof of (ii),  the reader is referred to the corresponding proof of Theorem \ref{resol_stand_test_strong}.
\end{proof}
\noindent
Finally, we present the theorem for the special case (\ref{weak_param}). 

\begin{theorem}
The reconstruction of Algorithm \ref{alg_lin_test_weak} fulfills the resolution guarantee if
\begin{align*}
\nu > 2\delta >0
\end{align*}
\noindent
with 
\begin{align*}
\nu:=\min_{s=1,\ldots,N}\left(\max\left(\mathrm{eig}\left[{\bf T}_s-{\boldsymbol\Lambda}(\lambda_{\mathcal{B}_{\min}},\mu_{\mathcal{B}_{\min}})\right]\right)\right).
\end{align*}
\end{theorem}

\begin{proof}
First, let $\omega_s\subseteq \mathcal{D}$ and let $g_j\in\mathbb{R}^M$, $j\in \lbrace 1, 2,\ldots ,N\rbrace$. In a body with interior Lam\'e parameters $(\lambda_{\mathcal{B}_{\max}},\mu_{\mathcal{B}_{\max}})$, let $u_{g_j}$ be the corresponding  displacements resulting from applying the boundary load $g_j$ to the $j$-th patch. As in the proof of the theorem before, we obtain
\begin{align*}
g_j^T\left({\boldsymbol\Lambda}(\lambda_{\mathcal{B}_{\max}},\mu_{\mathcal{B}_{\max}})-\delta {\bf I} -{\boldsymbol\Lambda}^\delta(\lambda,\mu)\right)g_j
\leq \int_{\mathcal{D}}\kappa^{\lambda}\nabla\cdot u_{g_j} \nabla\cdot u_{g_j} + 2\kappa^{\mu}\hat{\nabla} u_{g_j} : \hat{\nabla} u_{g_j}\,dx.
\end{align*}
\noindent
This yields
\begin{align*}
{\bf T}_s-\delta {\bf I} \leq {\boldsymbol\Lambda}^\delta(\lambda,\mu).
\end{align*}
\noindent
Hence, $\omega_s$ will be marked, which shows that part (i) of the resolution guarantee.  The second part is analogue to the proof of part (ii) from the theorem before.
\end{proof}

\subsection{Numerical simulations}

We examine an elastic body (Makrolon) with possible inclusions (aluminium), where the corresponding Lam\'e parameters are given in Table \ref{lame_parameter_mono}.

\begin{table} [H]
 \begin{center}
 \begin{tabular}{ |c|c| c |}  
\hline
 material & $\lambda_i$ & $\mu_i$ \\
  \hline
$i=0$: background material (Makrolon) &  $2.8910\cdot 10^9$   &  $1.1808\cdot 10^9$   \\
 \hline
$i=\mathcal{D}$: inclusion material (aluminium) &  $5.1084\cdot 10^{10}$ &  $2.6316\cdot 10^{10}$  \\
\hline
\end{tabular}
\end{center}
\caption{Lam\'e parameters of the test material in [Pa] (see \cite{Eberle_Experimental}).}
\label{lame_parameter_mono}
\end{table}

We consider two different settings of test cubes ($5\times 5\times 5$ and $10\times 10\times 10$)  as well as two configurations of Neumann patches. Specifically, we apply boundary forces on $5$ faces of the elastic body with either $5\times 5$ or $10\times 10$ Neumann patches on each face. Figure \ref{testcubes} shows exemplary the setting with $5\times 5\times 5$ testcubes and $125$ Neumann patches.

\begin{figure}[H]
\centering 
\includegraphics[width=0.5\textwidth]{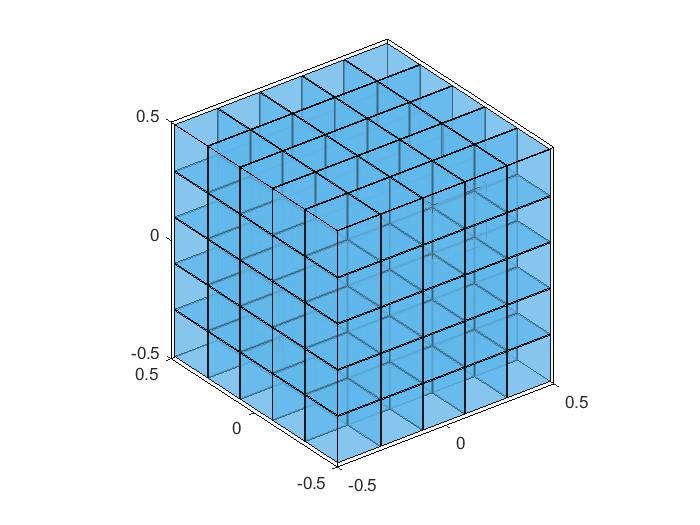}
\caption{$125$ testcubes and $125$ Neumann patches.}\label{testcubes}
\end{figure}

Our simulations are based on noisy data. We assume that we are given a noise level $\eta\geq 0$ and set
\begin{align*}
\delta=\eta\cdot \Vert \boldsymbol{\Lambda}(\lambda,\mu)\Vert_F.
\end{align*}
\noindent
In addition, we define $\boldsymbol{\Lambda}^\delta(\lambda,\mu)$ as
\begin{align*}
\boldsymbol{\Lambda}^\delta(\lambda,\mu)=\boldsymbol{\Lambda}(\lambda,\mu)+\delta \overline{{\bf E}},
\end{align*}
\noindent
with $\overline{{\bf E}}={\bf E}/ \Vert{\bf E}\Vert_F$, where ${\bf E}$ consists of $M\times M$ random uniformly distributed values in $[-1,1]$.

\subsubsection{Example 1}

For our simulations we calculate the maximal noise $\eta$ perturbing $\boldsymbol{\Lambda}(\lambda,\mu)$ for different background error parameters $\epsilon^\lambda$ and $\epsilon^\mu$
(see Figure \ref{5_N_5_T_2D}), based on Theorem \ref{cal_max}.

\begin{figure}[H]
\centering 
\includegraphics[width=0.45\textwidth]{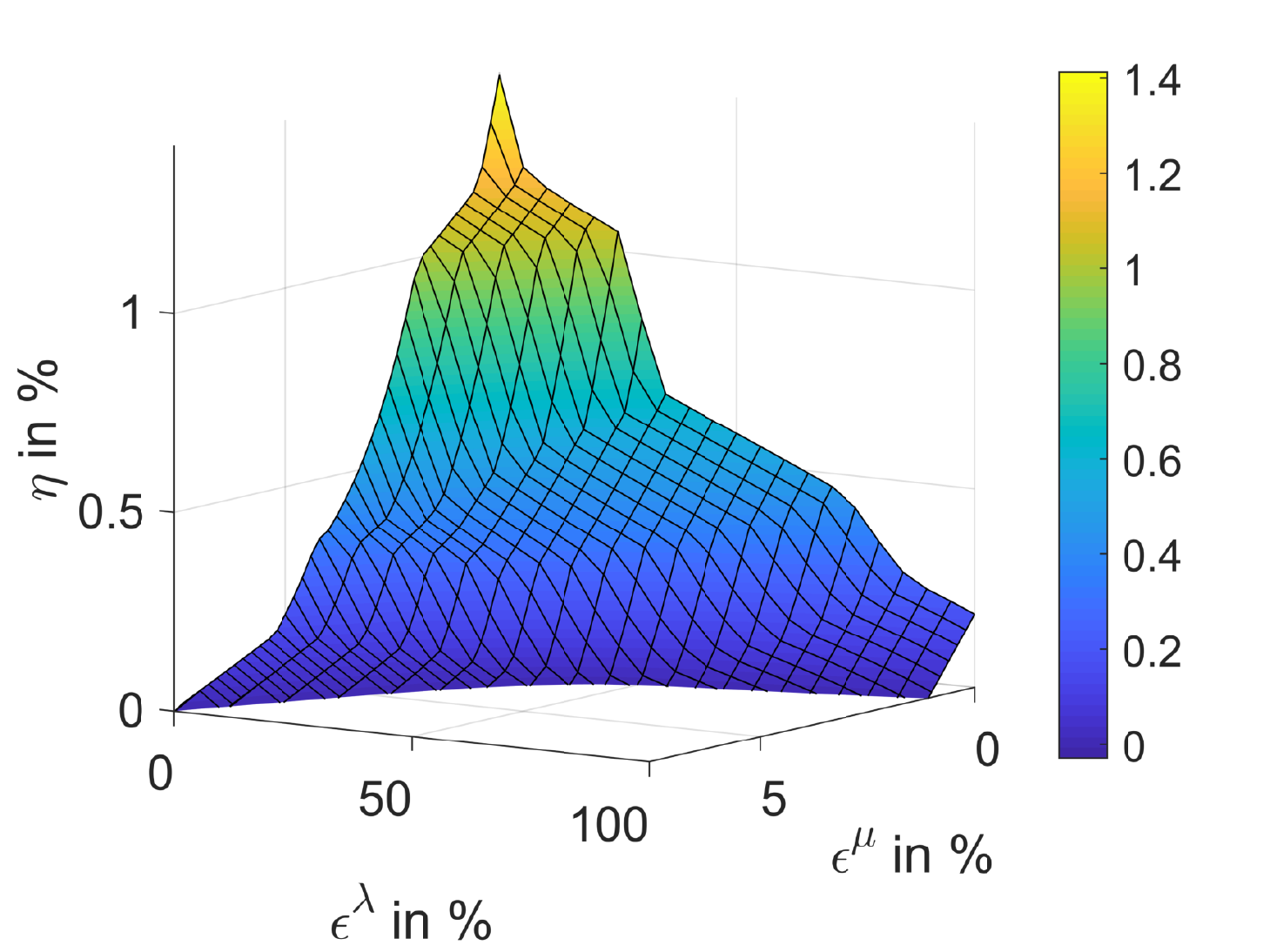}
\includegraphics[width=0.45\textwidth]{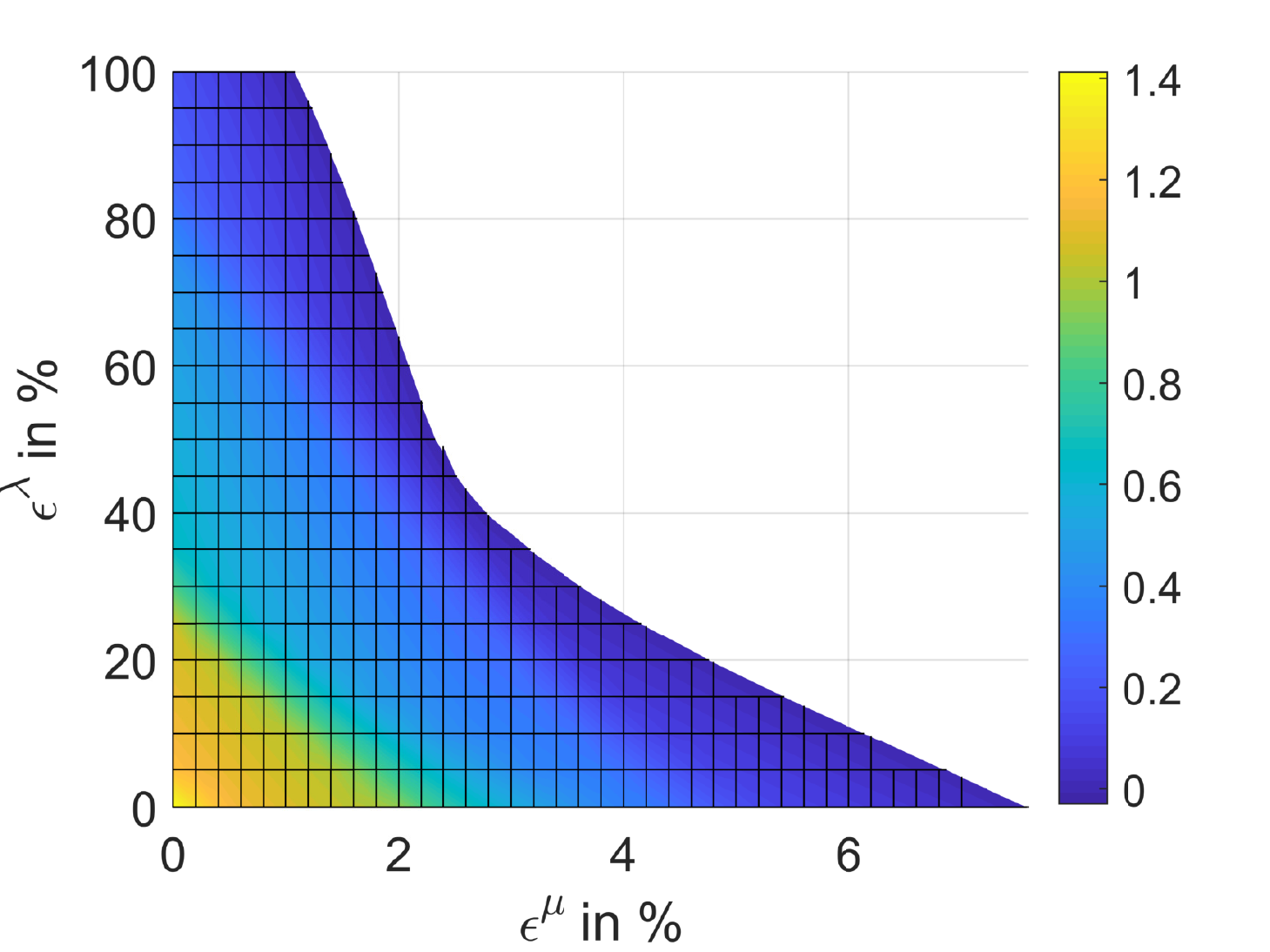}\\
\includegraphics[width=0.45\textwidth]{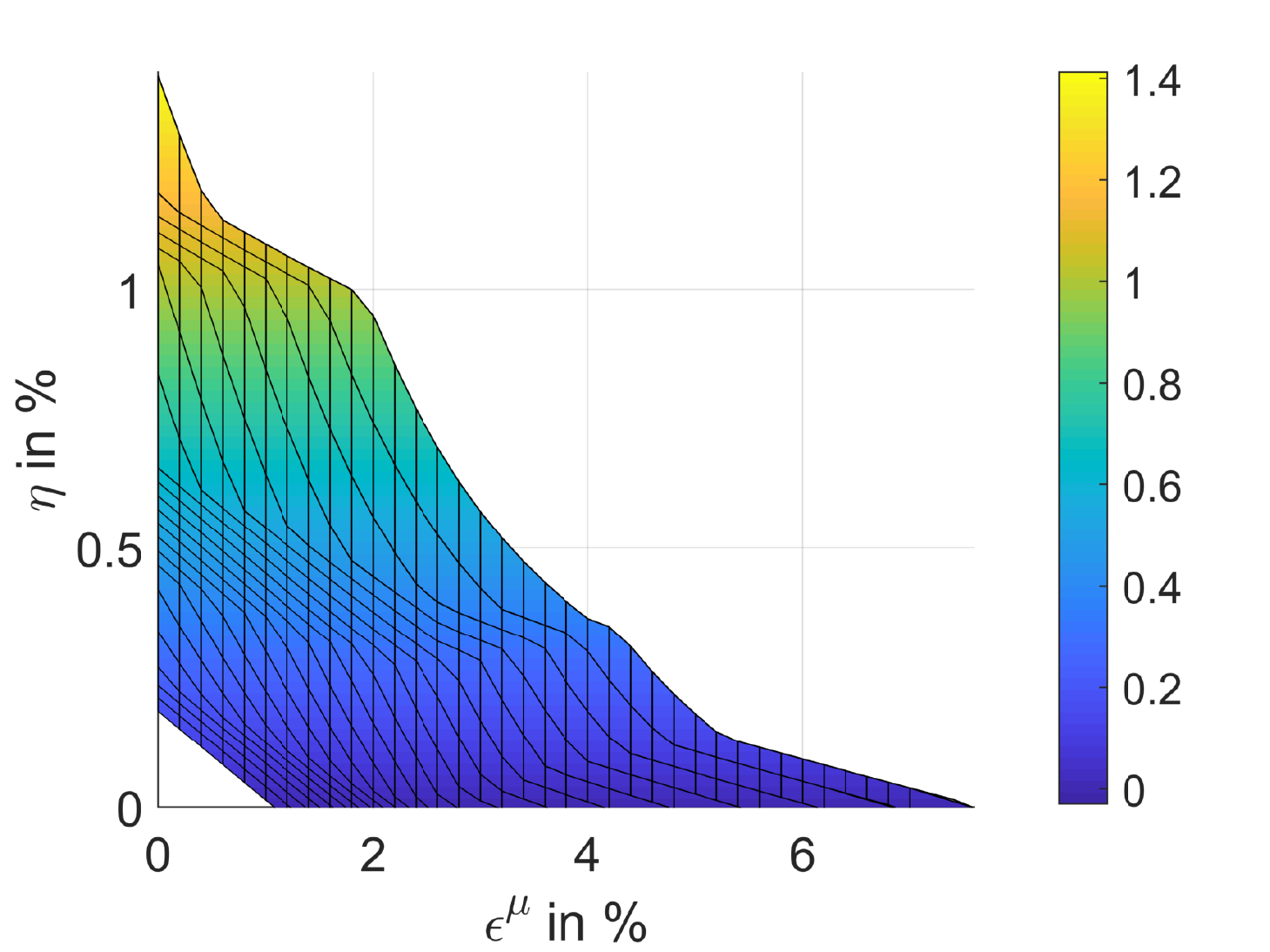}
\includegraphics[width=0.45\textwidth]{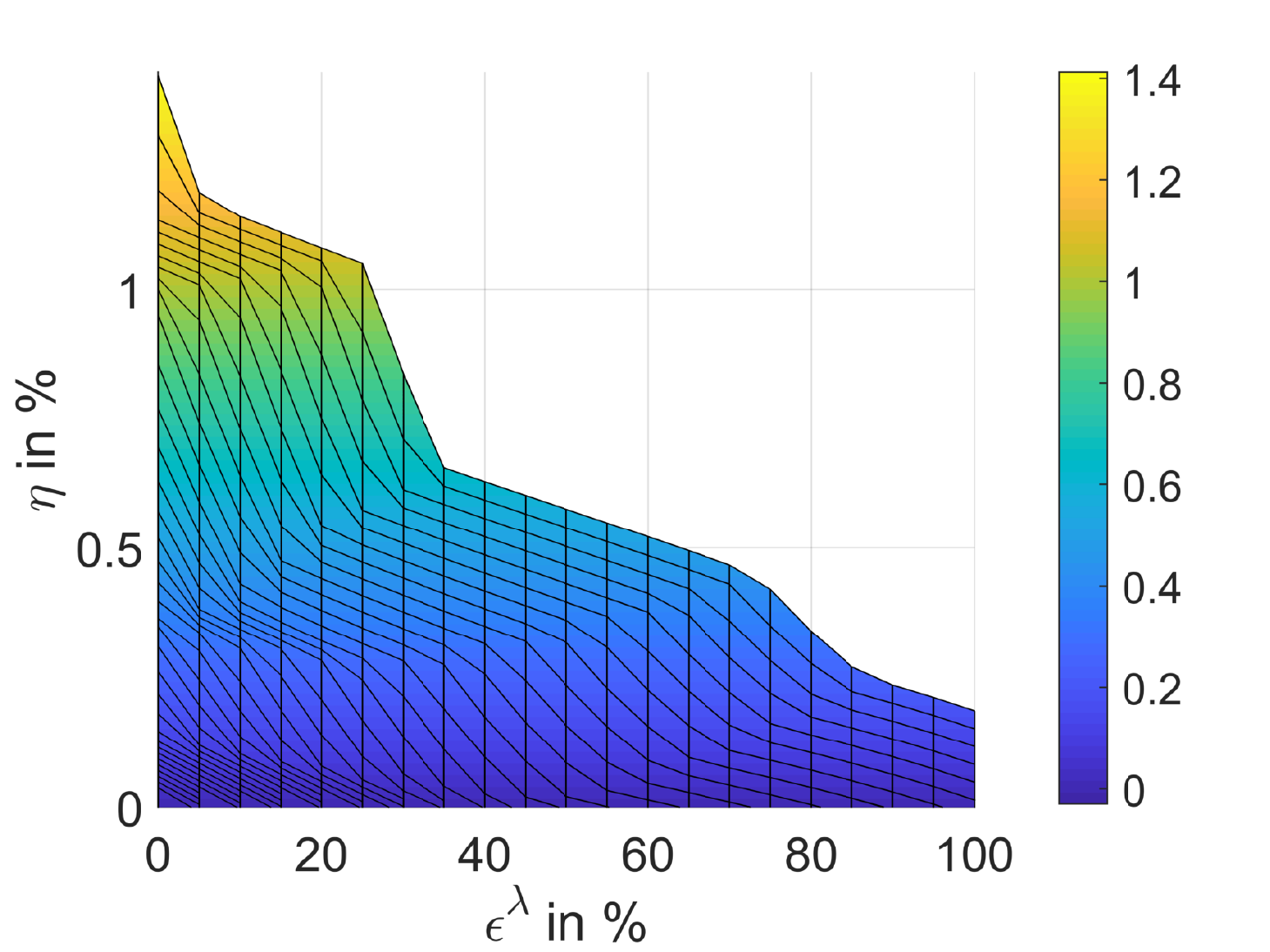}

\caption{Relation between $\eta$ and $\epsilon^\lambda$ as well as $\epsilon^\mu$ for $125$ testcubes and $125$ Neumann patches shown from different angles.}\label{5_N_5_T_2D}
\end{figure}
\noindent
\\
Figure \ref{5_N_5_T_2D} tells us that the maximal $\eta$ of approximately $1.413\%$ is reached for $\epsilon^\lambda=0=\epsilon^\mu$.  
The background error $\epsilon^\lambda$ does not show much impact. Even for $\epsilon^\lambda=100\%$, we obtain a resolution guarantee. The maximal background error w.r.t. $\mu$ with $\epsilon^\lambda=0\%$ is $\epsilon^\mu\approx 7.692\%$ at $\eta=0\%$.

\begin{remark}
All in all, we conclude that the resolution guarantees depend heavily on the Lam\'e parameter $\mu$
and only marginally on $\lambda$.
\end{remark}

\subsubsection{Example 2}

Based on the result of Example 1, we change our configuration and set $\epsilon^\lambda=0\%$ for a better comparability.
The results are shown in Figure \ref{5_N_5_T}-\ref{10_N_10_T}, where we analyse the relation of $\epsilon^\mu$ ($x$-axis) and $\eta$ ($y$-axis) with both values given in $\%$. The considered numbers of testcubes and Neumann patches are given in the caption of the figure. As expected, the smaller the background error $\epsilon^\mu$ can be estimated, the more noise on the data can be handled.
\\
\\
In Figure \ref{5_N_5_T}, we deal with $5\times 5\times 5=125$ testcubes and $125$ Neumann patches as shown in Figure \ref{testcubes}. We can observe an approximately linear connection between $\epsilon^\mu$ and $\eta$ showing that a resolution guarantee is given for all pairs $(\epsilon^\mu,\eta)$ on the black line and the gray area below for
$\epsilon^\lambda=0\%$.

\begin{figure}[H]
\centering 
\includegraphics[width=0.55\textwidth]{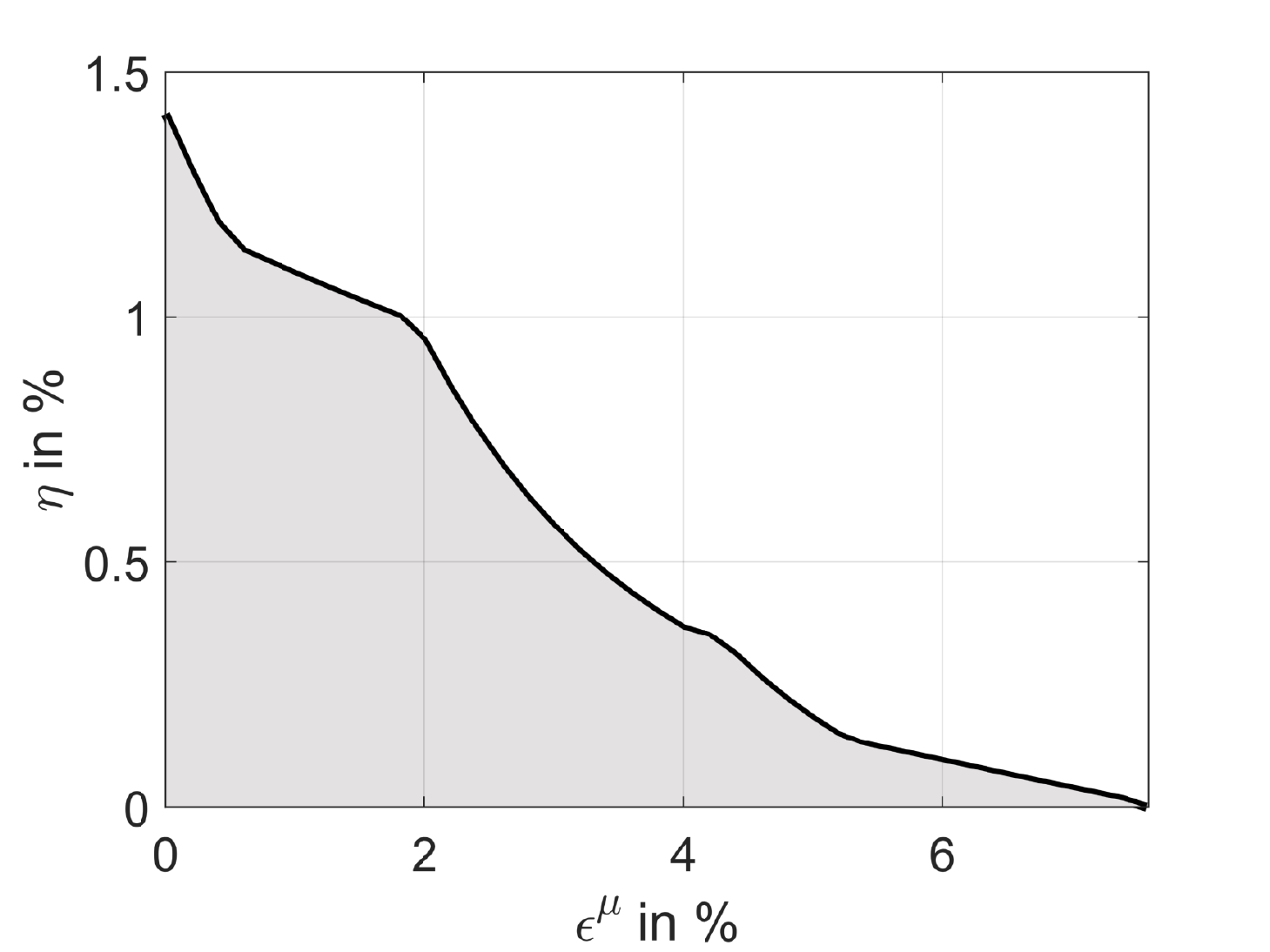}
\caption{Relation between $\eta$ and $\epsilon^\mu$ for $125$ testcubes and $125$ Neumann patches for $\epsilon^\lambda=0\%$.}\label{5_N_5_T}
\end{figure}
\noindent
In Figure \ref{5_N_10_T}, we change our setting and increase the number of testcubes to $10\times 10\times 10=1000$, while simulating the reconstruction for the same $125$ Neumann patches.
\\
\\
If we now compare Figure \ref{5_N_5_T} and \ref{5_N_10_T}, we see that for more testcubes, our method is less stable w.r.t. both $\epsilon^\mu$ and $\eta$. 
This behaviour is expected since smaller pixels are to be reconstructed with the same amount of data from the Neumann patches.
Nevertheless, we achieve a resolution guarantee, if the pair $\eta$, $\epsilon^\mu$ is located on the black line or the gray area below. The maximal noise on the date is given by $\eta\approx 0.200\%$ for $\epsilon^\mu=\epsilon^\lambda=0\%$ and the maximal background noise for $\mu$ is given by $\epsilon^\mu\approx 0.927\%$ for $\epsilon^\lambda=\eta=0\%$.

\begin{figure}[H]
\centering 
\includegraphics[width=0.55\textwidth]{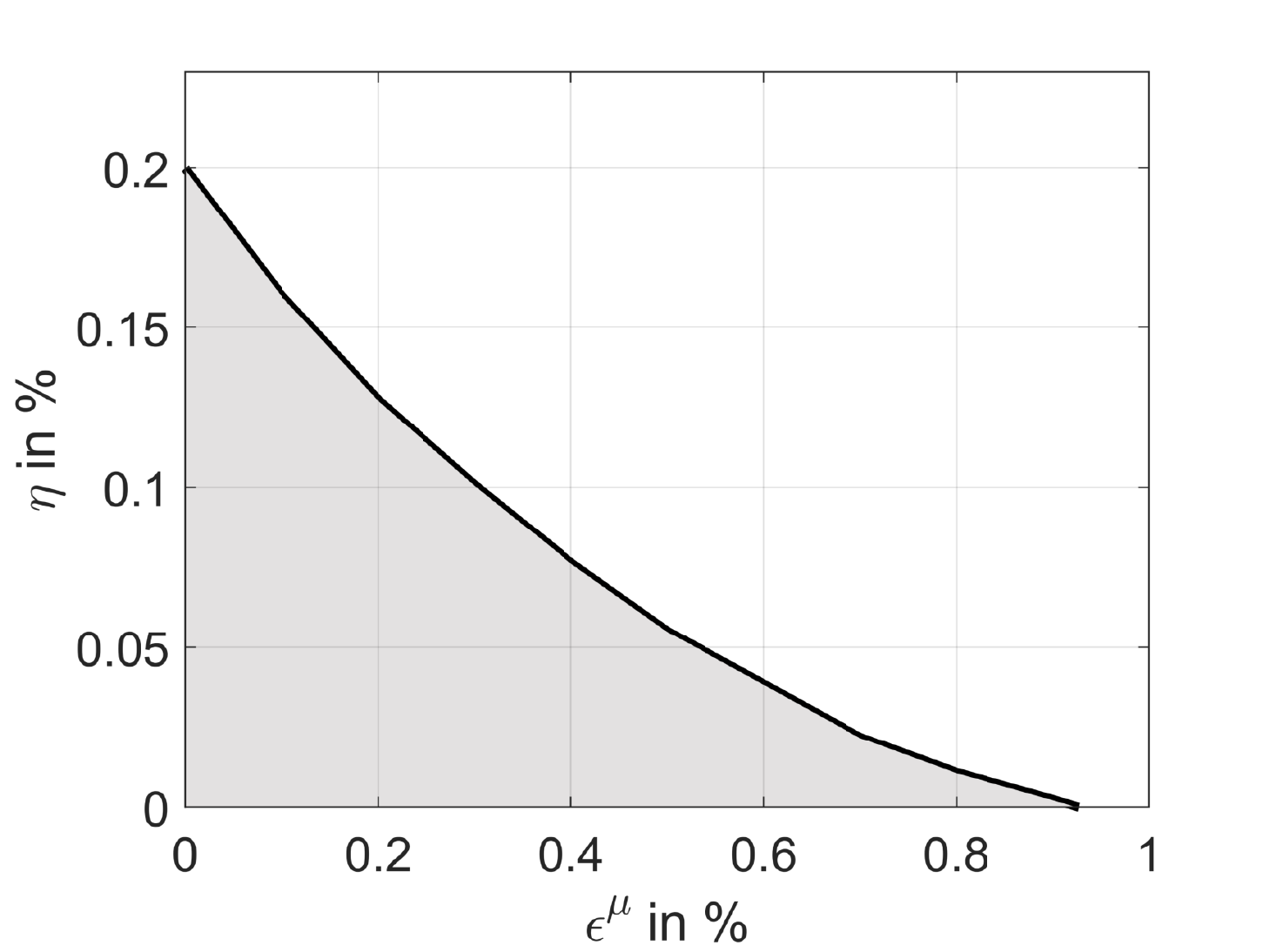}
\caption{Relation between $\eta$ and $\epsilon^\mu$ for $1000$ testcubes and $125$ Neumann patches for $\epsilon^\lambda=0\%$.}\label{5_N_10_T}
\end{figure}
\noindent
Increasing the resolution by using more Neumann patches is also possible as indicated in Figure \ref{10_N_10_T}. This figure shows the set-up with $1000$ testcubes, the same as in Figure \ref{5_N_10_T}, but with $500$ Neumann patches instead of $125$. This increases both the stability regarding $\eta$ as well es $\epsilon^\mu$, however, the improvement is small. 
In fact, the maximal noise on the date is given by $\eta\approx 0.213\%$ for $\epsilon^\mu=\epsilon^\lambda=0\%$ and the maximal background noise for $\mu$ is given by $\epsilon^\mu\approx 0.942\%$ for $\epsilon^\lambda=\eta=0\%$.
For a better resolution guarantee, even more Neumann patches have to be used, but the numerical effort to do that will increases heavily.

\begin{figure}[H]
\centering 
\includegraphics[width=0.55\textwidth]{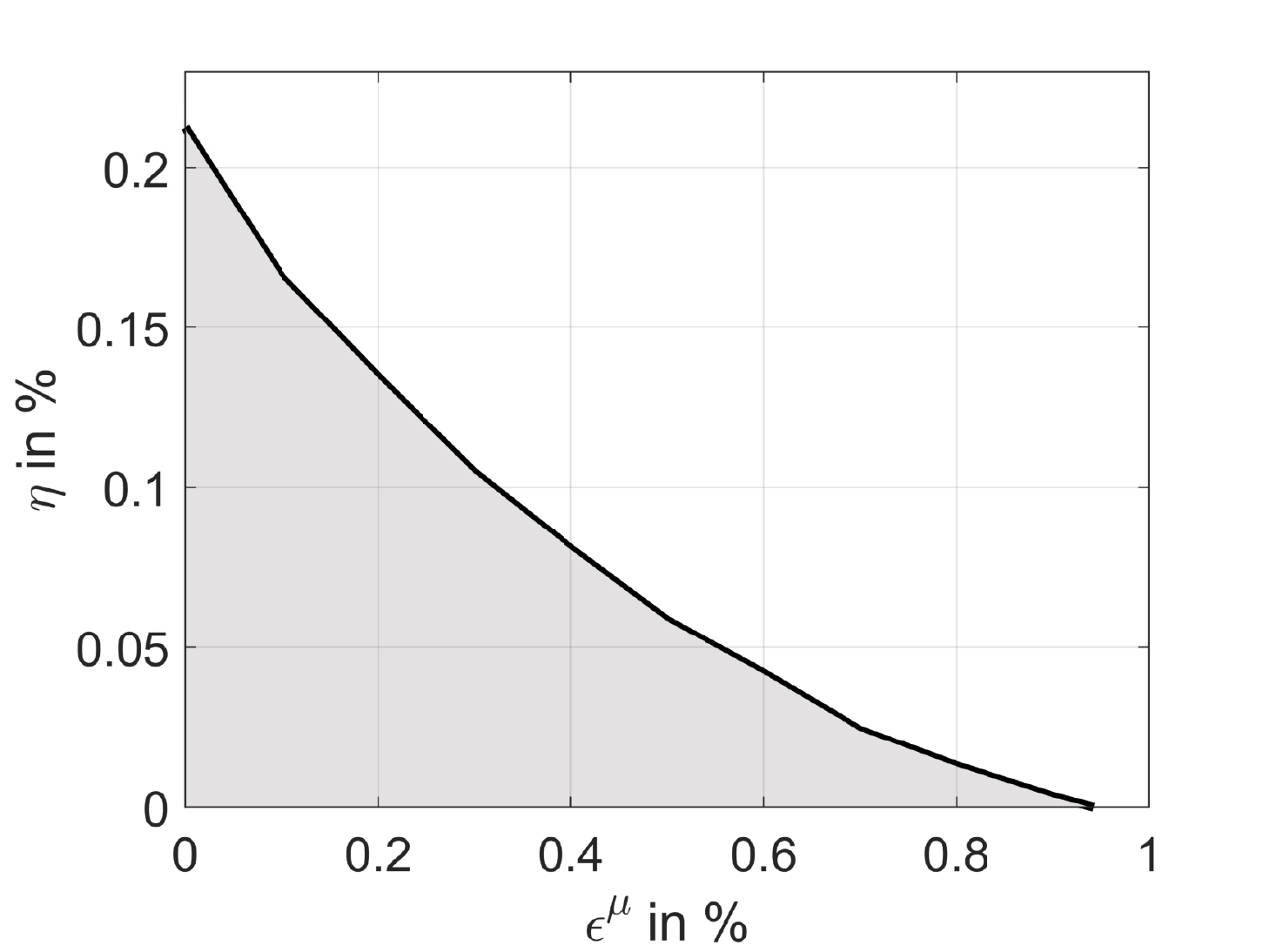}
\caption{Relation between $\eta$ and $\epsilon^\mu$ for $1000$ testcubes and $500$ Neumann patches for $\epsilon^\lambda=0\%$.}\label{10_N_10_T}
\end{figure}


\section{Conclusion and outlook}

Our main focus was the construction of conditions under which a resolution for a given partition can be achieved. 
Thus, our formulation takes both the background error as well as the measurement noise into account. The numerical simulations showed that for more testcubes our method is less stable  w.r.t. $\epsilon^\lambda, \epsilon^\mu$ and $\eta$. This behaviour is expected since more as well as smaller pixels are to be reconstructed with the same amount of data from the Neumann patches.
As a main
result, the resolution guarantees depend heavily on the Lam\'e parameter $\mu$
and only marginally on $\lambda$.
Finally, we want to remark that the algorithm is more stable w.r.t. $\epsilon^\lambda, \epsilon^\mu$ as w.r.t. $\eta$. 
All in all, our results are of special importance, when considering simulations based on real data, e.g., in \cite{Eberle_Experimental} or in the framework of monotonicity-based regularization (see, e.g. \cite{Eberle_Mon_Bas_Reg}).

\noindent
\\
{\bf Acknowledgements}
\\
The first author thanks the German Research Foundation (DFG) for funding the project "Inclusion Reconstruction with Monotonicity-based Methods for the Elasto-oscillatory Wave Equation" (reference number $499303971$).


\bibliographystyle{plain}
\bibliography{references}
\end{document}